\documentclass[a4paper,12pt]{article} 
 
\usepackage[parfill]{parskip}
\usepackage[titletoc,toc,title]{appendix}
\usepackage[margin=2cm]{geometry}
\usepackage{ mathrsfs }
\usepackage{amsthm}
\usepackage{amsmath}
\usepackage{amssymb}
\usepackage{authblk}
\usepackage{graphicx}
\usepackage{mathtools}
\usepackage{subcaption}
\usepackage{enumerate}
\usepackage{enumitem}
\usepackage{xcolor}
\usepackage{tikz}
\usetikzlibrary{calc,intersections,through,backgrounds}
\usepackage{array}
\usepackage{tikz-cd}
\usepackage{mathabx}
\usepackage[bookmarks=true]{hyperref}
\hypersetup{
	colorlinks,
	linkcolor={red!70!black},
	citecolor={red!90!black},
	urlcolor={green!80!black},
	linktoc=all, 
}

\makeatletter
\newtheorem*{rep@theorem}{\rep@title}
\newcommand{\newreptheorem}[2]{%
	\newenvironment{rep#1}[1]{%
		\def\rep@title{#2 \ref{##1}}%
		\begin{rep@theorem}}%
		{\end{rep@theorem}}}
\makeatother
\newtheoremstyle{stylename}
{15pt} 
{15pt} 
{\itshape} 
{} 
{\bfseries} 
{.} 
{.5em} 
{} 
\theoremstyle{stylename}
\newtheorem{theorem}{Theorem}[section]
\newreptheorem{theorem}{Theorem}
\newtheorem{lemma}[theorem]{Lemma}

\newtheoremstyle{exampstyle}
{15pt} 
{15pt} 
{} 
{} 
{\bfseries} 
{.} 
{.5em} 
{} 
\theoremstyle{exampstyle} 
\newtheorem{definition}[theorem]{Definition}   
\newtheorem{example}[theorem]{Example}
\newtheorem{remark}[theorem]{Remark}

\usepackage{titlesec}

\titleformat{\chapter}[display]
  {\normalfont\sffamily\huge\bfseries\color{black}}
  {\chaptertitlename\ \thechapter}{20pt}{\Huge}

\titleformat{\section}
  {\Large\bfseries\color{black}}
  {\thesection}{1em}{}

\usepackage[backend=bibtex,style=alphabetic,sorting=nyt,maxbibnames=99,giveninits=true]{biblatex}  
\renewbibmacro{in:}{} 
\DeclareFieldFormat{pages}{#1}
\DeclareFieldFormat[article]{number}{ no. #1}
\bibliography{bib2} 

\begin{document}


\title {A family of flat  Minkowski planes over convex functions}
\author{ Duy Ho}
\affil{UAE University, PO Box 15551, Al Ain, United Arab Emirates}
\affil{duyho92@gmail.com}
	\date{ }

	\maketitle
\begin{abstract} Using suitable convex functions, we construct a new family of flat Minkowski planes whose automorphism groups are at least $3$-dimensional. These planes admit  groups of automorphisms isomorphic to the direct product of  $\mathbb{R}$ and the connected component of the affine group on $\mathbb{R}$.  We also determine isomorphism classes, automorphisms and possible Klein-Kroll types for our examples.
\end{abstract}


  
\section{Introduction}

Flat Minkowski planes are generalization of the geometry of nontrivial plane sections of the standard nondegenerated ruled quadric in the real 3-dimensional projective space $\mathbb{P}_3(\mathbb{R})$.  A flat Minkowski plane $\mathcal{M}$ can be classified based on the dimension $n$ of its automorphism group.  It is known that  $n$ is at most $6$, and  the plane $\mathcal{M}$ is determined when $n \ge 4$, cp.  \cite{schenkel1980}. The current open case of interest is when $n=3$, in which a list of possible connected groups of automorphisms of $\mathcal{M}$ was presented in \cite{brendan2017b}.

The purpose of this paper is to describe a new family of flat Minkowski planes that admits the 3-dimensional connected group
$$
\Phi_\infty= \{(x,y) \mapsto (x+b,ay+c) \mid a,b,c \in \mathbb{R}, a>0  \}
$$ 
as their groups of automorphisms.   
The construction presented here was motivated from a family of flat Laguerre planes of translation type  described by L\"{o}wen and Pf\"{u}ller \cite{lowen1987a}, \cite{lowen1987b}.
These Laguerre planes were constructed as follows: first a suitable convex function $f$ called ``strongly parabolic function" is defined, then  circles are generated as images of $f$ under the group $\Phi_\infty$.   In our construction, we introduce the notion of a ``strongly hyperbolic function'', which is adapted from that of strongly parabolic functions to satisfy the incidence axioms of flat Minkowski planes.

The paper is organized as follows. Section 2 contains the preliminary results and examples. In Section 3, we define strongly hyperbolic functions and derive necessary properties. In Section 4, we describe  a family of flat Minkowski planes using strongly hyperbolic functions. 
The paper ends with a discussion on isomorphism classes, automorphisms and Klein-Kroll types of these planes in Section 5. 
\section{Preliminaries}
\subsection{Flat Minkowski planes and two halves of the circle set}
A \textit{flat Minkowski plane} is a geometry $\mathcal{M}=(\mathcal{P}, \mathcal{C}, \mathcal{G}^+, \mathcal{G}^-)$, whose 
\begin{enumerate}[label=$ $]
	\item point set $\mathcal{P}$ is  the torus $\mathbb{S}^1 \times \mathbb{S}^1$, 
	\item circles (elements of $\mathcal{C}$) are graphs of homeomorphisms of $\mathbb{S}^1 $,  
	\item $(+)$-parallel classes (elements of $\mathcal{G}^+$) are the verticals $\{ x_0 \}  \times \mathbb{S}^1$,  
	\item $(-)$-parallel classes (elements of $\mathcal{G}^-$) are the horizontals $\mathbb{S}^1 \times \{ y_0 \}$,  
\end{enumerate} 
where $x_0, y_0 \in \mathbb{S}^1$.  We denote the $(\pm)$-parallel class containing a point $p$ by $[p]_\pm$. When two points $p$ and $q$ are on the same $(\pm)$-parallel class, we say they are \textit{$(\pm)$-parallel}  and denote  this by $p \parallel_{\pm} q$. Two points $p,q$ are $\textit{parallel}$ if they are  $(+)$-parallel or $(-)$-parallel, and we denote this by $p \parallel q$. 
Furthermore, a flat Minkowski plane satisfies the following two axioms:
\begin{enumerate}[label=]
	\item \textit{Axiom of Joining}: three pairwise nonparallel points  $p,q,r$ can be joined by a unique circle.  
	\item \textit{Axiom of Touching}: for each circle $C$ and any two nonparallel points $p,q$ with $p \in C$ and $q \not \in C$, there is exactly one circle $D$ that contains both points $p,q$ and intersects $C$ only at the point $p$.
\end{enumerate}


The \textit{derived plane $\mathcal{M}_p$ of $\mathcal{M}$ at the point $p$} is the incidence geometry whose point set is $\mathcal{P} \backslash ( [p]_+ \cup [p]_-)$, whose lines are all parallel classes not going through $p$ and all circles of $\mathcal{M}$ going through $p$.   
For every point $p \in \mathcal{P}$, the derived plane $\mathcal{M}_p$ is  a flat  affine plane, cp. \cite[Theorem 4.2.1]{gunter2001}. 

A flat Minkowski plane is in \textit{standard representation} if the set $\{(x,x)\mid x \in \mathbb{S}^1\}$ is one of its circles  (cp. \cite[ Subsection 4.2.3]{gunter2001}). Up to isomorphisms, every flat Minkowski plane can be described in standard representation. In this case, we omit the two parallelisms and refer to $(\mathcal{P}, \mathcal{C})$ as a flat Minkowski plane.

Let $\mathcal{M}=(\mathcal{P},\mathcal{C})$ be a flat Minkowski plane in standard representation. Let $\mathcal{C}^+$ and $\mathcal{C}^-$ be the sets of all circles in $\mathcal{C}$ that are graphs of orientation-preserving and orientation-reversing homeomorphisms of $\mathbb{S}_1$, respectively. Clearly, $\mathcal{C}=\mathcal{C}^+ \cup \mathcal{C}^-$. We call $\mathcal{C}^+$ and $\mathcal{C}^-$ the \textit{positive half} and \textit{negative half} of $\mathcal{M}$, respectively. It turns out that these two halves are completely independent of each other, that is, we can exchange halves from different flat Minkowski planes and obtain another flat Minkowski plane, see \cite[Subsection 4.3.1]{gunter2001}:

\begin{theorem} \label{twohalves} For $i=1,2$, let $\mathcal{M}_i=(\mathcal{P},\mathcal{C}_i)$ be two flat Minkowski planes. Then $\mathcal{M}= (\mathcal{P}, \mathcal{C}_1^+ \cup \mathcal{C}_2^-)$ is a flat Minkowski plane. 
\end{theorem}

\subsection{Examples}

 We identify $\mathbb{S}^1$ with $\mathbb{R} \cup \{ \infty \}$ in the usual way. There are various known examples of flat Minkowski planes, cp. \cite[Chapter 4]{gunter2001}. For our purposes, we recall three particular examples.

\begin{example} \label{ex:classical} The circle set of the \textit{classical flat Minkowski plane $\mathcal{M}_C$} consists of sets of the form 
	$
	\{(x,sx+t)\mid x \in \mathbb{R} \} \cup \{ (\infty,\infty)\},
	$
	where $s,t \in \mathbb{R}$, $s \ne 0$, and
	$
	\{ (x,y) \in \mathbb{R}^2 \mid (x-b)(y-c)=a \} \cup \{ (\infty,c),(b,\infty) \},
	$
	where $a,b,c \in \mathbb{R}$, $a \ne 0$.
 \end{example}

%
%

\begin{example}  \label{ex:hartmann}
	For $r,s>0$, let $f_{r,s}$ be the orientation-preserving  semi-multiplicative homeomorphism of $\mathbb{S}^1$ defined by
	$$f_{r,s}(x) = 
	\begin{cases}
	x^r   & \text{for } x\ge 0,\\
	-s|x|^r & \text{for } x<0, \\
	\infty  & \text{for } x=\infty. \\
	\end{cases} 
	$$
	The circle set  of a \textit{generalised Hartmann plane $\mathcal{M}_{GH}(r_1,s_1;r_2,s_2)$} consists of sets of the form 
	$$
	\{(x,sx+t)\mid x \in \mathbb{R} \} \cup \{ (\infty,\infty)\},
	$$
	where $s,t \in \mathbb{R}$, $s \ne 0$,   sets of the form 
	$$
	\left\{ \left(x,\dfrac{a}{f_{r_1,s_1}(x-b)}+c \right) \;\middle|\;  x \in \mathbb{R} \right\}  \cup \{ (b,\infty),(\infty,c)\},
	$$
	where $a,b,c \in \mathbb{R}$, $a > 0$, and  sets of the form
	$$
	\left\{ \left(x,\dfrac{a}{f_{r_2,s_2}(x-b)}+c \right) \;\middle|\;  x \in \mathbb{R}  \right\}  \cup \{ (b,\infty),(\infty,c)\},
	$$
	where $a,b,c \in \mathbb{R}$, $a < 0$. 
\end{example}

\begin{example} Let $f$ and $g$ be two orientation-preserving homeomorphisms of $\mathbb{S}^1$. Denote $\text{\normalfont PGL}(2,\mathbb{R})$ by $\Xi$ and $\text{\normalfont PSL}(2,\mathbb{R})$ by $\Lambda$. The circle set $\mathcal{C}(f,g)$ of a \textit{half-classical plane $\mathcal{M}_{HC}(f,g)$} consists of sets of the form 
	$
	\{ (x,\gamma(x)) \mid x \in \mathbb{S}^1 \},
	$
	where $\gamma \in \Lambda \cup g^{-1}(\Xi \backslash \Lambda) f$. 
\end{example}
\subsection{Automorphisms and Klein-Kroll types}
  
An \textit{isomorphism between two flat Minkowski planes} is a bijection between the point sets that maps circles to circles, and induces a bijection between the circle sets. 
An \textit{automorphism of a   flat Minkowski plane $\mathcal{M}$} is an isomorphism from $\mathcal{M}$ to itself.  Every automorphism of a flat Minkowski plane is continuous and thus a homeomorphism of the torus. With respect to composition, the set of all automorphisms of a flat Minkowski plane is a group called the automorphism group $\text{Aut}(\mathcal{M})$ of $\mathcal{M}$.  The group $\text{Aut}(\mathcal{M})$ is a Lie group of dimension at most $6$ with respect to the compact-open topology.  We say  a flat Minkowski plane has \textit{group dimension n} if its automorphism group has dimension $n$. For $n \ge 4$, we have the following classification, cp. \cite{schenkel1980}.
 
\begin{theorem} \label{introtheorem2} Let $\mathcal{M}$  be a flat Minkowski plane with group dimension $n$. If $n \ge 4$, then exactly one of the following occurs.
	
	\begin{enumerate}
		\item $n=6$ and $\mathcal{M}$ is isomorphic to the classical flat Minkowski plane.

		\item $n=4$ and $\mathcal{M}$ is isomorphic to a  proper (nonclassical) generalised Hartmann plane $\mathcal{M}_{GH}(r_1,s_1;r_2,s_2)$, $r_1,s_1,r_2,s_2 \in \mathbb{R}^+$, $(r_1,s_1,r_2,s_2) \ne (1,1,1,1)$. 
			\item $n=4$ and $\mathcal{M}$ is isomorphic to a proper half-classical plane $\mathcal{M}_{HC}(f,id)$, where $f$ is a semi-multiplicative homeomorphism of $\mathbb{S}_1$ of the form $f_{d,s}$, $(d,s) \ne (1,1)$.
	\end{enumerate}
\end{theorem}
   
  A \textit{central automorphism}  of a Minkowski plane is an automorphism that fixes at least one point and induces a central collineation in the derived projective plane at each fixed point.
Similar to the Lenz-Barlotti classification of projective planes with respect to central collineations, Minkowski planes have been classified by Klein and Kroll with respect to groups of central automorphisms
that are ‘linearly transitive’, cp. \cite{klein1992}  and \cite{kleinkroll1989}.
In the case of flat Minkowski planes, possible  Klein-Kroll types  were determined  by Steinke \cite{gunter2007} as follows.  

\begin{theorem} \label{possiblekleinkroll} A flat Minkowski plane has Klein-Kroll type  
	\begin{enumerate}[label=$ $]
		\item  \text{\normalfont\begin{tabular}{  r l  }  
				I. &  A.1, A.2, A.3, B.1, B.10, B.11, D.1, \\
				II. &  A.1, A.15, \\
				III. & C.1, C.18, C.19, \\
				IV.&  A.1, \textit{or} \\
				VII.&  F.23.
		\end{tabular} }
	\end{enumerate}
	
\end{theorem} 
For each of these 14 types, except type II.A.15, examples are given in \cite{gunter2007}.  
In the same paper, Steinke also characterised  some families of flat Minkowski planes.
The following  result is adapted from \cite[Proposition 5.9]{gunter2007}. 
\begin{theorem} \label{specifickleinkroll}	A flat Minkowski plane of Klein-Kroll type
	\begin{enumerate}[label=$ $]
		
		\item \text{\normalfont VII.F.23} is isomorphic to the classical flat Minkowski plane;
		
		\item \text{\normalfont III.C.19} is isomorphic to a proper generalised Hartmann plane $\mathcal{M}_{GH}(r, 1; r, 1), r \ne 1$;
		
		\item \text{\normalfont III.C.18} is isomorphic to an   Artzy-Groh plane  (cp. \cite{artzy1986}) with group dimension $3$.
	\end{enumerate}
\end{theorem}
 
\section{Strongly hyperbolic functions} 
In this section, we will define strongly hyperbolic functions and derive some properties. 
For convenience, we will abbreviate the mean value theorem by MVT, and the intermediate value theorem by IVT.

\begin{definition} \label{stronglyhyperbolic}
	A function $f: \mathbb{R}^+ \rightarrow \mathbb{R}^+$ is called a \textit{strongly hyperbolic function} if it satisfies the following conditions.
	\begin{enumerate}
		\item $ \lim\limits_{x \rightarrow  0+} f(x)= +\infty$ and $\lim\limits_{x \rightarrow  +\infty} f(x)= 0$. 
		\item $f$ is strictly convex. 
		
		\item For each $b \in \mathbb{R}$, $$
		\lim\limits_{x \rightarrow  +\infty}  \frac{f(x+b)}{f(x)}=1.  $$
		\item $f$ is  differentiable.
		\item $\ln|f'(x)|$ is strictly convex.
	\end{enumerate}
\end{definition}

\subsection{Properties strongly hyperbolic  functions}
 
Let $f$ be a strongly hyperbolic function. It is readily checked that $f$ is strictly decreasing and consequently an orientation-reversing homeomorphism of $\mathbb{R}^+$. 
Since $f$ is strictly convex and differentiable, it is continuously differentiable. We now  consider some other properties of $f$.

%
%
%
%

\begin{lemma} \label{stronglyhyperboliclimits}  \label{derivativelnf} Let $b>0$, $s,t \ne 0$. Then the following statements are true.
	\begin{enumerate}
		\item
		$$\lim\limits_{x \rightarrow \infty} \dfrac{f'(x)}{f'(x+b)}=1.$$
		\item
		$$
		\lim\limits_{x \rightarrow  +\infty} \frac{f'(x)}{f(x)} = 0.
		$$
		\item
		$$
		\lim\limits_{x \rightarrow  +\infty} \frac{f(x+s)-f(x)}{f'(x)} =s. 
		$$
		\item
		$$
		\lim\limits_{x \rightarrow  +\infty} \frac{f(x+s)-f(x)}{f(x+t)-f(x)} = \frac{s}{t}.
		$$
		\item 
			$$
		\liminf_{x \rightarrow 0+}  \dfrac{f'(x)}{f(x)} = -\infty.
		$$
		
	\end{enumerate}
	\begin{proof}
		\begin{enumerate}[leftmargin=0pt,itemindent=*]
			\item 	Let  $h(x)=\ln(|f'(x)|)$. 	Since $f$ is strongly hyperbolic, $h$ is strictly convex and strictly decreasing. We define $h_b: (0,\infty) \rightarrow \mathbb{R}$ by 
			$$
			h_b(x)= h(x)-h(x+b)= \ln \left( \left|\dfrac{f'(x)}{f'(x+b)} \right| \right).
			$$
			Then $h_b$ is strictly decreasing and  is bounded below by $0$.
			This implies $ \lim\limits_{x \rightarrow \infty} h_b(x) $ exists and so does $\lim\limits_{x \rightarrow \infty} \dfrac{f'(x)}{f'(x+b)}$. The proof now follows from L'Hospital's Rule. 
			
			\item Let $\delta>0$. For $x>\delta$, by the MVT, there exists $r \in (x-\delta,x)$ such that
			$$
			f'(r)= \dfrac{f(x)-f(x-\delta)}{\delta}.
			$$
			Since $f$ is strictly convex, $f'$ is strictly increasing, so that $f'(r)< f'(x).$ Dividing both sides by $f(x)$, we get
			$$
			\dfrac{f(x)-f(x-\delta)}{\delta f(x)}<\dfrac{f'(x)}{f(x)}<0.
			$$
			
			Since $f$ is strongly hyperbolic, 
			$$
			\lim\limits_{x \rightarrow +\infty} \dfrac{f(x)-f(x-\delta)}{f(x)}= 1- \lim\limits_{x \rightarrow +\infty} \dfrac{f(x-\delta)}{f(x)}=1-1=0.
			$$
			
			The claim now follows from the squeeze  theorem.
			\item We will assume $s>0$, as the case $s<0$ is similar.
			For $x>0$,  by the MVT, there exists $r \in (x,x+s)$ such that
			$$
			f'(r) = \frac{f(x+s)-f(x)}{s}. 
			$$
			Since $f'$ is strictly increasing and negative, $
			f'(x) < f'(r) < f'(x+s)<0.
			$
			Hence,
			$$
			sf'(x) < f(x+s)-f(x)  <sf'(x+s)<0,
			$$ 
			and so
			$$
			\frac{sf'(x+s)}{f'(x)}  <  \frac{f(x+s)-f(x)}{f'(x)}  < s.
			$$
			The claim now follows from  the squeeze theorem and part 1.
			
			\item We rewrite 
			$$
			\frac{f(x+s)-f(x)}{f(x+t)-f(x)} = \frac{f(x+s)-f(x)}{f'(x)} \cdot \frac{f'(x)}{f(x+t)-f(x)}.
			$$
			The claim now follows from part 3. 
			
			\item This can be derived from the fact that $\lim\limits_{x \rightarrow 0+} \ln f(x) = +\infty$. 
			\qedhere

		\end{enumerate}
	\end{proof}
\end{lemma}

%
%
%
%
%
%
%

%

We now  study the roots of the function 
$\widecheck{f}: (\max \{ -b,0\},\infty) \rightarrow \mathbb{R}$ defined as
$$\widecheck{f}(x)=af(x+b)+c-f(x),$$
where   $a>0,b,c \in \mathbb{R}$. 
We first consider the derivative $\widecheck{f}'$. 

\begin{lemma} \label{derivativestronglyhyperbolic}  \label{continuouslydifferentiable} 
	The derivative $\widecheck{f}'$ is continuous and has at most one root. Furthermore, if $\widecheck{f}'$ has a root $x_0$, then $\widecheck{f}'$ changes sign at $x_0$.

	\begin{proof}    Since $\ln|f'(x)|$ is strictly convex and strictly decreasing, the function $h: (\max \{ -b,0\},\infty) \rightarrow \mathbb{R}$ defined by
		$$ h(x)= \ln|f'(x+b)|+ \ln a - \ln |f'(x)|$$
		is strictly monotonic. Hence, $h$ has at most one root, and if $h$  has a root $x_0$, then $h$ changes sign at $x_0$. 
		Note that $\widecheck{f}'$ has a root $x_0$ if and only if $h$ has $x_0$ as a root, and  $\widecheck{f}'(x)\lessgtr0$ if and only if $h(x)\gtrless 0$. This completes the proof.	\qedhere
		
	\end{proof}
\end{lemma}
%

We note that if $a \ne 1$ and $b=c=0$,	then  $\widecheck{f}$ has no roots.
In the following  lemma, we consider the special case when exactly one of $b,c$ is zero. If $a=1$, then  $\widecheck{f}$ has no roots. We then further assume $a \ne 1$. 

\begin{lemma} \label{T2M2}
	 If $a \ne 1$  and either  $b\ne 0,c=0$ or $b=0,c\ne 0$, then $\widecheck{f}$
	has at most one root.  
	Furthermore,   exactly one of the following statements is true. 
	
	\begin{enumerate}
		
		\item $\widecheck{f}$ has  exactly one root $x_0$ at which it changes sign. The derivative $\widecheck{f}'$ is nonzero at $x_0$, and either
		$$a>1,b>0,c=0 \text{ or } a>1,b=0,c<0 \text{ or } a<1,b<0,c=0 \text{ or } a<1,b=0,c>0.$$
		\item $\widecheck{f}$ has no root, and either 
		$$a<1,b>0,c=0 \text{ or } a<1,b=0,c<0 \text{ or } a>1,b<0,c=0 \text{ or } a>1,b=0,c>0.$$
	\end{enumerate} 
	
	\begin{proof} There are four cases depending on the sign of $b,c$. We only prove the cases $b>0,c=0$ and $b=0,c>0$. The cases $b<0,c=0$ and $b=0,c<0$ are similar. 
		\begin{enumerate}[label=Case  \arabic*:,leftmargin=0pt,itemindent=*]
			\item $b >0, c=0$. We have
			$\widecheck{f}(x)=af(x+b)-f(x).$ If $a<1$, then
			$$ \widecheck{f}(x)<f(x+b)-f(x)<0,$$
			and so $\widecheck{f}$ has no roots. 
			We now claim that if $a>1$, then $\widecheck{f}$ has  exactly one root $x_0$ at which it changes sign. 
			Let $g: (0,+\infty) \rightarrow \mathbb{R}$ be defined by 
			$$g(x)= \ln f(x+b) + \ln a - \ln f(x).$$

			We note that $\widecheck{f}$ has a root if and only if $g$ has a root, and the sign of $\widecheck{f}(x)$ is the same as the sign of $g(x)$. 
			Also, $g$ is continuous, $\lim\limits_{x \rightarrow 0} g(x) = -\infty$, and $\lim\limits_{x \rightarrow +\infty} g(x) =\ln a >0$.  
			It follows that $g$, and in particular $\widecheck{f}$, has at least one root. 
			
			By Lemma \ref{derivativestronglyhyperbolic}, $\widecheck{f}$ cannot have more than two roots. Suppose for a contradiction that $\widecheck{f}$ has exactly two roots $x_0<x_1$. By Lemma \ref{derivativestronglyhyperbolic} and Rolle's Theorem, $\widecheck{f}'$ is nonzero at these roots. Since $\widecheck{f}'$ is continuous, $\widecheck{f}$ is locally monotone at the roots. Hence $\widecheck{f}$ changes sign at the two roots.  It follows that $g$ has two roots at which it changes sign. 
			
			Since $\lim\limits_{x \rightarrow 0} g(x) = -\infty$, $g(x)<0$ for $x \in (0,x_0)$. Since $g$ changes sign at $x_0$ and has no roots between $x_0$ and $x_1$, we have $g(x)>0$ for $x \in (x_0,x_1)$. Since $g$ changes sign at $x_1$ and has no roots larger than $x_1$, we have $g(x)<0$ for $x>x_1$. This contradicts $\lim\limits_{a \rightarrow +\infty} g(x) =\ln a >0$.  
			
			Therefore $\widecheck{f}$ has exactly one root $x_0$. Then $g$ also has exactly one root at $x_0$. By the IVT, $g$ changes sign at $x_0$. Then $\widecheck{f}$ also changes sign at $x_0$.  
			This proves the claim.
			
			\item $b =0, c >0$. Then 
			$
			\widecheck{f}(x) = (a-1)f(x)+c.$
			If $a>1$, then $\widecheck{f}>0$ and has no roots. If $a<1$, then from Definition \ref{stronglyhyperbolic},  $\widecheck{f}$ has exactly one root at which it changes sign. \qedhere

		\end{enumerate}
	\end{proof}
	
\end{lemma}
We now consider the case when both $b,c \ne 0$. 
\begin{lemma} \label{T1meat}
If $b,c \ne 0$, then $\widecheck{f}(x)$ has at most two roots. Furthermore,   exactly one of the following statements is true. 
	
	\begin{enumerate}
		
		\item  $\widecheck{f}$ has exactly  two roots $x_0$ and $x_1$ at which it changes sign. The derivative $\widecheck{f}'$ is nonzero at $x_0$ and $x_1$, and either  
		$$a<1,b<0,c>0 \text{ or } a>1,b>0,c<0.$$ 
		\item $\widecheck{f}$ has  exactly  one root $x_0$ at which it does not change sign. The derivative $\widecheck{f}'$ also has a root at $x_0$,  and either  
		$$a<1,b<0,c>0 \text{ or } a>1,b>0,c<0.$$
		\item $\widecheck{f}$ has exactly one   root $x_0$ at which it changes sign. The derivative $\widecheck{f}'$ is nonzero at $x_0$, and $bc >  0$.

		\item $\widecheck{f}$ has  no roots, and $bc<0$.
		
	\end{enumerate}	
	
	\begin{proof}    \begin{enumerate}[leftmargin=0pt,itemindent=*]
			
			\item From Lemma \ref{derivativestronglyhyperbolic} and Rolle's Theorem, $\widecheck{f}$ cannot have more than two roots. Also, if 
			$\widecheck{f}$ has exactly two roots, then  $\widecheck{f}'$ is nonzero at these roots. Moreover, since $\widecheck{f}$ is continuously differentiable, it must change sign at the two roots. 
			
			Assume $\widecheck{f}$ has exactly one root $x_0$. If $\widecheck{f}'(x_0) \ne 0$, then   $\widecheck{f}$ changes sign at $x_0$. In the case $\widecheck{f}(x_0)=\widecheck{f}'(x_0)=0$, by Lemma \ref{derivativestronglyhyperbolic}, the derivative $\widecheck{f}'$ changes sign at $x_0$. This implies $\widecheck{f}$ has a local extremum at $x_0$ and so does not change sign.
			
			If none of the above occurs, then it must be the case that $\widecheck{f}$ has no roots. Therefore, excluding the conditions on the parameters $a,b,c$, exactly one of the statements in the lemma is true. 
			
			\item We claim that if $bc>0$, then $\widecheck{f}$ has exactly one root at which $\widecheck{f}$ changes sign. 
			We assume that $b>0,c>0$; the case $b<0,c<0$ is similar. Since 
			$\lim\limits_{x \rightarrow 0} \widecheck{f}(x)= -\infty,$ and $\lim\limits_{x \rightarrow +\infty} \widecheck{f}(x)=c$, by the IVT, $\widecheck{f}$ has at least one root. 
			
			 Suppose for a contradiction that $\widecheck{f}$ has exactly two roots $x_0<x_1$.  Since $\lim\limits_{x \rightarrow 0} \widecheck{f}(x) = -\infty$, $\widecheck{f}(x)<0$ for $x \in (0,x_0)$. Since $\widecheck{f}$ changes sign at $x_0$ and has no roots between $x_0$ and $x_1$, we have $\widecheck{f}(x)>0$ for $x \in (x_0,x_1)$. Since $\widecheck{f}$ changes sign at $x_1$ and has no roots larger than $x_1$, we have $\widecheck{f}(x)<0$ for $x>x_1$. This contradicts $\lim\limits_{x \rightarrow +\infty} \widecheck{f}(x) =c >0$.  This proves the claim.

			\item Assume $\widecheck{f}$ has exactly  two roots $x_0$ and $x_1$ at which it changes sign. From part 2 of the proof and Lemma \ref{T2M2}, it must be the case that $bc<0$. Assume $b>0,c<0$. We rewrite 
			$$\widecheck{f}(x)=a\left(f(x+b)-f(x)\right) +c+(a-1)f(x).$$
			Since $f$ is strictly convex, $f(x+b)-f(x)$ is strictly increasing. If $0<a\le 1$, then $\widecheck{f}$ is   strictly increasing and thus has at most one root, which contradicts our assumption. Therefore $a>1$. 
			Similarly, when  $b<0,c>0$, we have $a<1$.
			
			\item   Assume $\widecheck{f}$ has exactly one   root at which it changes sign. We can say there exists $x^*$ such that $\widecheck{f}(x^*)>0$.  Suppose for a contradiction that $b>0,c<0$. Then 
			$\lim\limits_{x \rightarrow 0} \widecheck{f}(x)= -\infty,$ and $\lim\limits_{x \rightarrow +\infty} \widecheck{f}(x)=c<0$. By the IVT, there exist two roots $x_0 \in (0,x*)$ and $x_1 \in (x^*,\+\infty)$, which contradicts our assumption. Similarly, it cannot be the case $b<0,c>0$. Thus, $bc>0$. 
			
			\item When $\widecheck{f}$ has  exactly  one root $x_0$ at which it does not change sign, the conditions on $a,b,c$  follow in a similar manner as in part 3).  When $\widecheck{f}$ has no roots, then by part 2 of the proof, $bc<0$. This completes the proof.  \qedhere

			%
			%
			%
			%
			%
			%
			%
			
		\end{enumerate}
	\end{proof}
\end{lemma}  
\subsection{Graphs of strongly hyperbolic functions under  $\Phi_\infty$}
Let $a_1,a_2>0$, $b_1,b_2,c_1,c_2 \in \mathbb{R}$. For two strongly hyperbolic functions $f_1$ and $f_2$, we define 
$$\widecheck{f}: (\max \{ -b_1,-b_2\},+\infty) \rightarrow \mathbb{R}: x \mapsto a_1f_1(x+b_1)+c_1-a_2f_1(x+b_2)-c_2,$$  
$$\widehat{f}: (-\infty,\min \{ -b_1,-b_2\}) \rightarrow \mathbb{R}: x \mapsto -a_1f_2(-x-b_1)+c_1+a_2f_2(-x-b_2)-c_2.$$ 
From the previous subsection, we obtain the following two lemmas. 
\begin{lemma}  \label{M2full}
	
	Assume  $b_1\ne b_2,c_1=c_2$ or $b_1=b_2,c_1 \ne c_2$. 
%
%
	
	\begin{enumerate}

		\item   If $a_1=a_2$, then both $\widecheck{f}$ and  $\widehat{f}$ have no roots.	
		\item  If $a_1\ne a_2$, then either $\widecheck{f}$ or $\widehat{f}$, but not both, has  exactly one root $x_0$ at which it changes sign.  The corresponding derivative  is nonzero at $x_0$.
		
	\end{enumerate}
%
%
%
%
%
%
%
%
\end{lemma}

\begin{lemma} \label{obviouscalculation} 
	

%
	Assume $b_1\ne b_2, c_1 \ne c_2$.	If  $\widecheck{f}$ has at least one root, then exactly one of the following is true.
	\begin{enumerate}
		\item  $\widecheck{f}$ has exactly  two roots $x_0$ and $x_1$ at which it changes sign, $\widehat{f}$ has no roots. 
		The derivative $\widecheck{f}'$ is nonzero at $x_0$ and $x_1$.

		\item  $\widecheck{f}$ has exactly one root $x_0$ at which it does not change sign,  $\widehat{f}$ has no roots. The derivative $\widecheck{f}'$ also has a root at $x_0$.
		
		\item $\widecheck{f}$ has exactly one root $x_0$ at which it changes sign,  $\widehat{f}$ also has exactly one root $x_1$ at which it changes sign. The derivatives $\widecheck{f}'(x_0)$ and $\widehat{f}'(x_1)$ are nonzeros.
	\end{enumerate}

%
\end{lemma}

\subsection{Some examples and remarks} 
\begin{example} \label{exampleSH}	The following functions $f  :\mathbb{R}^+ \rightarrow \mathbb{R}^+$ are strongly hyperbolic.
	\begin{enumerate}
		\item
		$
		f(x)=\dfrac{1}{x^i}, 
		$
		where $i \in \mathbb{N}$. 
		\item 
	$
	f(x)=\sum\limits_{i=1}^n  \dfrac{1}{x^i},
	$
		where $n \in \mathbb{N}$. 
	\item  
	$
	f(x)=\dfrac{1}{x+\arctan(x)}.
	$
	\item 
	$
f(x)= \ln\left(\dfrac{1}{x}+\sqrt{\dfrac{1}{x^2}+1}\right).
	$
\end{enumerate}
\end{example}

\begin{remark} The inverse of a  strongly hyperbolic function is not necessarily a strongly hyperbolic function. 
	For instance, the inverse of the last function $f$ in Example \ref{exampleSH}  is given by $$f^{-1}(x)=\dfrac{1}{\sinh(x)}.$$

	As mentioned in \cite{hartman1981}, the inverse $f^{-1}$ is not strongly hyperbolic, since, for $b \ne 0$,
	$$
	\lim\limits_{x \rightarrow \infty} \dfrac{\sinh(x)}{\sinh(x+b)} = e^{-b} \ne 1.
	$$ 
%
%
%
%
%
%
\end{remark}
%

%

\begin{remark} The definition of strongly hyperbolic functions share some similarities with the definition of strongly parabolic functions in the construction of flat Laguerre planes of translation type by L\"{o}wen and Pf\"{u}ller \cite{lowen1987a}. We note that strongly parabolic functions are assumed to be twice differentiable. As mentioned by the authors in \cite[Remark 2.6]{lowen1987a}, this condition was proved to be unnecessary by Schellhammer \cite{schellhammer1981}. For the definition of strongly hyperbolic functions, we omit this condition. 
	\end{remark}

\section{Flat Minkowski planes from strongly hyperbolic functions}
%
%
	Let $f_1,f_2$ be strongly hyperbolic functions.  For $a>0,b,c \in\mathbb{R}$, let  	
	\begin{align*}
	f_{a,b,c}: \mathbb{R} \backslash \{ -b \} \rightarrow  \mathbb{R} \backslash \{ c\} : x \mapsto 
	\begin{cases}
	af_1(x+b)+c  & \text{for } x>-b,\\
	-af_2(-x-b)+c & \text{for } x<-b.
	\end{cases}
	\end{align*}  
	
We also define the following sets:
	\begin{align*}
	\overline{f_{a,b,c}}		&\coloneq \{ (x,f_{a,b,c}(x))\mid x\in\mathbb{R} \backslash \{-b\}  \} \cup \{(-b,\infty),(\infty,c)\},  \\
	F		& \coloneq \{ \overline{f_{a,b,c}}\mid a>0,b,c\in \mathbb{R} \},   \\
	\overline{l_{s,t}}&\coloneq \{(x,sx+t)\mid x\in \mathbb{R} \} \cup \{(\infty,\infty)\},  \\ 
	L&\coloneq \{\overline{l_{s,t}} \mid s,t \in \mathbb{R}, s < 0\}.
	\end{align*}   

%
	For a set  $\overline{f_{a,b,c}}$, 
  \textit{the convex branch of $\overline{f_{a,b,c}}$} is   the subset
	$
	\{ (x,f_{a,b,c}(x)) \mid x>-b\}$, and \textit{the concave branch of $\overline{f_{a,b,c}}$} is the subset 
	$
	\{ (x,f_{a,b,c}(x))\mid x<-b\}.
	$  
	
%
%

We define $\mathcal{C}^-(f_1,f_2)\coloneq F \cup L$ and $ \mathcal{C}^{+}(f_1,f_2)\coloneq \varphi(\mathcal{C}^{-}(f_1,f_2))$, where  $\varphi$ is the homeomorphism of the torus defined by $\varphi: (x,y) \mapsto (-x,y).$  
 In this section, we prove the following theorem.    
\begin{theorem} \label{main}
For $i=1..4$, let $f_i$ be a strongly hyperbolic function. Let  $\mathcal{C} \coloneq  \mathcal{C}^{-}(f_1,f_2) \cup \mathcal{C}^{+}(f_3,f_4)$. Then  
$\mathcal{M}_f=\mathcal{M}(f_1,f_2;f_3,f_4):=(\mathcal{P}, \mathcal{C})$ is a flat Minkowski plane. 
\end{theorem} 

In view of Theorem \ref{twohalves}, to prove Theorem \ref{main} it is sufficient to prove that $\mathcal{C}^-(f_1,f_2)$ is the negative half of a flat Minkowski plane. We verify that this is the case by showing that $\mathcal{C}^-(f_1,f_2)$ satisfies  Axiom of Joining and Axiom of Touching in the next four subsections.

%

\subsection{Axiom of Joining, existence} 
 Three points $p_1,p_2,p_3 \in \mathcal{P}$ are in \textit{admissible position} if they can be joined by an element of the negative half of the classical flat Minkowski plane $\mathcal{M}_C$.  Furthermore, we say they are of type
	
	\begin{enumerate} [label=\arabic*] 
		
		\item if $p_1= (\infty,\infty),p_2,p_3 \in \mathbb{R}^2$,
		
		\item if $ p_1= (x_1,y_1) ,p_2= (\infty,y_2), p_3= (x_3,\infty),$ $x_i,y_i \in \mathbb{R}$,
		
		\item if $ p_1= (x_1,y_1),p_2=(x_2,y_2), p_3 = (x_3, \infty),$ $x_i,y_i \in \mathbb{R}$,
		
		\item if $p_1= (x_1,y_1),p_2=(x_2,y_2), p_3 = (\infty, y_3) $, $x_i,y_i \in \mathbb{R}$,
		
		\item if $p_1= (x_1,y_1),p_2=(x_2,y_2), p_3 = (x_3, y_3), x_i,y_i \in \mathbb{R}$. 	
	\end{enumerate}
 
We note that up to permutations, if three points are in admissible position, then they are in exactly one of the five admissible position types. The main theorem of this  subsection is the following.
 
\begin{theorem}[Axiom of Joining, existence] \label{existencejoining} Let $p_1,p_2,p_3 \in \mathcal{P}$ be three points in admissible  position.  Then there is at least one  element in 	$\mathcal{C}^{-}(f_1,f_2)$ that contains $p_1,p_2,p_3$.
\end{theorem}
We prove Theorem \ref{existencejoining} in the following five lemmas. 

\begin{lemma} \label{ejointype1} Let $p_1,p_2,p_3$ be three points in admissible position type 1. Then there exist $s_0<0, t_0 \in \mathbb{R}$ such that   $\overline{l_{s_0,t_0}}$ contains $p_1,p_2,p_3$.
	\begin{proof} Since $\mathcal{C}^{-}(f_1,f_2)$ contains lines with negative slope extended by the point $(\infty,\infty)$, the choice of $\overline{l_{s_0,t_0}}$ is the line containing $p_2,p_3$.
	\end{proof}
\end{lemma}

\begin{lemma} \label{ejointype2} Let $p_1,p_2,p_3$ be three points in admissible position type 2. Then there exist  $a_0>0,b_0,c_0\in \mathbb{R}$   such that $\overline{f_{a_0,b_0,c_0}}$ contains $p_1,p_2,p_3$. 
	\begin{proof} Without loss of generality, we can assume $y_2=x_3=0$ so that $p_2=(\infty,0), p_3=(0,\infty)$. If $\overline{f_{a_0,b_0,c_0}}$ contains $(\infty,0), (0,\infty)$, then   $b_0=c_0=0$.  Since the points $p_1,p_2,p_3$ are in admissible position, there are two cases depending on $x_1,y_1$. 
		\begin{enumerate}[label=Case  \arabic*:,leftmargin=0pt,itemindent=*]
			
			\item 	$x_1,y_1>0$. The equation
			$
			y_1=af_1(x_1)
			$
			has a solution $a_0=\dfrac{y_1}{f_1(x_1)}$, and we have $\overline{f_{a_0,b_0,c_0}}$ whose convex branch contains $p_1$.
			
			\item $x_1,y_1<0$.  The equation
			$
			y_1=af_1(x_1)
			$
			has a solution $a_0=-\dfrac{y_1}{f_2(-x_1)}$, and we have $\overline{f_{a_0,b_0,c_0}}$ whose concave branch contains $p_1$. \qedhere
		\end{enumerate}
	\end{proof}
\end{lemma}

\begin{lemma} \label{ejointype3} Let $p_1,p_2,p_3$ be three points in admissible position type 3. Then there exist  $a_0>0,b_0,c_0\in \mathbb{R}$   such that  $\overline{f_{a_0,b_0,c_0}}$ contains $p_1,p_2,p_3$. 
	
	\begin{proof}  We can assume $x_3=0,y_2=0,x_1>x_2$. 
		 Let $b_0=0$. 
		There are three cases depending on the positions of $p_1$ and $p_2$.  
			We show that there is a choice of $a_0,c_0$ in each case.
		\begin{enumerate}[label=Case  \arabic*:,leftmargin=0pt,itemindent=*]
			\item $x_1>x_2>0$, and $y_1<0$. We consider the system 
			\begin{equation*} 
			\begin{cases} 
			y_1= af_1(x_1)+c\\
			0= af_1(x_2)+c
			\end{cases}
			\end{equation*} 
			in variables $a,c$. The solution for this system is 
			$
			a_0=\dfrac{y_1}{f_1(x_1)-f_1(x_2)}          
			$	 and  $c_0=-a_0f_1(x_2)$.  Then $\overline{f_{a_0,b_0,c_0}}$ contains $p_1$ and $p_2$  on its convex branch.   
			\item $x_1>0>x_2,$ and $y_1>0$. We claim that there exists $\overline{f_{a_0,b_0,c_0}}$ whose convex branch contains $p_1$ and whose concave branch contains $p_2,$  that is, the system
			\begin{equation*} 
			\begin{cases} 
			y_1= af_1(x_1)+c\\
			0= -af_2(-x_2)+c
			\end{cases}
			\end{equation*} 
			has a solution $a_0,c_0$. It is sufficient to show that the function $g: \mathbb{R}^+ \rightarrow \mathbb{R}$  defined by $$g(a)= (f_1(x_1)+f_2(-x_2))a-y_1,$$ has a root $a_0$, which is immediate from the IVT.  
			
			\item  $0>x_1>x_2,$ and $y_1<0$. One can show that there exists $\overline{f_{a_0,b_0,c_0}}$ whose concave branch contains $p_1$ and $p_2.$ This is similar to Case 1. \qedhere
		\end{enumerate}
	\end{proof} 
\end{lemma}
 
\begin{lemma} \label{ejointype4} Let $p_1,p_2,p_3$ be three points in admissible position type 4. Then   there exist  $a_0>0,b_0,c_0\in \mathbb{R}$   such that   $\overline{f_{a_0,b_0,c_0}}$ contains $p_1,p_2,p_3$. 
	
	\begin{proof}  We can assume $y_3=0, x_2=0, y_1>y_2$. Let $c_0=0$.  There are three cases.
		
		\begin{enumerate}[label=Case  \arabic*:,leftmargin=0pt,itemindent=*]
			
			\item $y_1>y_2>0$, and $x_1<0.$ We show that there exists $\overline{f_{a_0,b_0,c_0}}$ whose convex branch contains $p_1$ and $p_2,$ by showing that the system
			\begin{equation} \label{ejointype4eqn1}
			\begin{cases} 
			y_1= af_1(x_1+b)\\
			y_2= af_1(b)
			\end{cases}
			\end{equation} 
			has a solution $a_0,b_0$. Eliminating the variable $a$, we get
			\begin{equation} \label{ejointype4eqn2}
			\dfrac{y_1}{y_2}=\dfrac{f_1(x_1+b)}{f_1(b)}.
			\end{equation}
			We consider the function $g: (-x_1,+\infty) \rightarrow \mathbb{R}$ defined by $g(b) = \dfrac{f_1(x_1+b)}{f_1(b)}$. We have that $g(b)$ continuous, $\lim\limits_{b \rightarrow +\infty} g(b) = 1$ by Definition \ref{stronglyhyperbolic}, and $\lim\limits_{b \rightarrow -x_1+} g(b)= +\infty$. By the IVT, there exists $b_0$ such that 
			$g(b_0) =\dfrac{y_1}{y_2}>1$, that is,  $b_0$ is a root of (\ref{ejointype4eqn2}). This shows that (\ref{ejointype4eqn1}) also has a solution.
			
			\item $y_1>0>y_2$, and $x_1>0$. It can be shown that there exists $\overline{f_{a_0,b_0,c_0}}$ whose convex branch contains $p_1$ and whose concave branch contains $p_2$.  This is similar to Case 1. 
%
			
			\item $0>y_1>y_2$, and $x_1<0$. It can be shown that there exists $\overline{f_{a_0,b_0,c_0}}$ whose concave branch contains $p_1$ and $p_2$. This is similar to Case 1. \qedhere
			
		\end{enumerate}
	\end{proof}

\end{lemma}

\begin{lemma} \label{ejointype5} Let $p_1,p_2,p_3$ be three points in admissible position type 5. Then  either there exist  $a_0>0,b_0,c_0\in \mathbb{R}$  such that   $\overline{f_{a_0,b_0,c_0}}$ contains $p_1,p_2,p_3$;  or,   there exist  $s_0<0,t_0\in \mathbb{R}$   such that   $\overline{l_{s_0,t_0}}$ contains $p_1,p_2,p_3$. 
	\begin{proof} We can assume $x_3=y_3=0$ and $0<x_1<x_2$. For nontriviality, we further assume the three points are not collinear. There are four cases depending on the values of $y_1$ and $y_2$. 
		\begin{enumerate}[label=Case  \arabic*:,leftmargin=0pt,itemindent=*]
			
			\item  $y_2<y_1<0$ and $y_1<\dfrac{x_1}{x_2}y_2<0$.  We claim  that there exists $\overline{f_{a_0,b_0,c_0}}$ whose convex branch contains $p_1, p_2, p_3,$  by showing that the system
			$$ 
			\begin{cases}
			y_1= af_1(x_1+b)+c\\
			y_2= af_1(x_2+b)+c\\
			0= af_1(b)+c
			\end{cases}
			$$
			has a solution $a_0,b_0,c_0$. The third equation gives $c=-af_1(b)$. Substitute this into the first two equations and rearrange, we get 
			\begin{equation} \label{ejointype5eqn1}
			a= \dfrac{y_1}{f_1(x_1+b)-f_1(b)}= \dfrac{y_2}{f_1(x_2+b)-f_1(b)}.
			\end{equation}
			We consider the function $g: (0,+\infty) \rightarrow \mathbb{R}$ defined by
			$$
			g(b)= \dfrac{f_1(x_2+b)-f_1(b)}{f_1(x_1+b)-f_1(b)}=1+\dfrac{f_1(x_2+b)-f_1(x_1+b)}{f_1(x_1+b)-f_1(b)} .
			$$
		We note that $g$ is continuous and  $\lim\limits_{b \rightarrow 0+} g(b) =1$. Also, $\lim\limits_{b \rightarrow +\infty} g(b) = \dfrac{x_2}{x_1}$ by Lemma \ref{stronglyhyperboliclimits}. The conditions of $x_i,y_i$ imply that $1<\dfrac{y_2}{y_1}<\dfrac{x_2}{x_1}$. By the IVT, there exists $b_0$ such that $g(b_0)=\dfrac{y_2}{y_1}$. This shows that the second equality in (\ref{ejointype5eqn1}) has a solution and the claim follows.

			\item $y_1>y_2>0$. It can be shown that exists $\overline{f_{a_0,b_0,c_0}}$ whose convex branch contains $p_1, p_2$ and whose concave branch contains  $p_3.$  This is similar to Case 1. 
			\item $y_1<0$, and $y_2>0$. It can be shown that  there exists $\overline{f_{a_0,b_0,c_0}}$ whose concave branch contains $p_1, p_3$ and whose convex branch contains $p_2$. This is similar to Case 1.

			\item $y_2<y_1<0$ and $\dfrac{x_1}{x_2}y_2<y_1<0$. It can be shown that  there exists $\overline{f_{a_0,b_0,c_0}}$ whose concave branch contains $p_1, p_2,p_3$. This is similar to Case 1. \qedhere
			
		\end{enumerate}
		
	\end{proof}
\end{lemma}

\subsection{Axiom of Joining, uniqueness}	
In this subsection we prove the following.
\begin{theorem}[Axiom of Joining, uniqueness] \label{uniquenessjoining} Two distinct elements $C,D \in \mathcal{C}^{-}(f_1,f_2)$ have at most two intersections.
\end{theorem}
\begin{proof} We note that the theorem holds if at least one of $C$ or $D$ has the form $\overline{l_{s,t}}$.  In the remainder of the proof, let $C=\overline{f_{a_1,b_1,c_1}}$ and $D=\overline{f_{a_2,b_2,c_2}}$, where  $a_1,a_2>0, b_1,b_2,c_1,c_2 \in \mathbb{R}, (a_1,b_1,c_1) \ne (a_2,b_2,c_2)$.
We assume that $C$ and $D$ have two intersections $p,q$ and show that they can have no other intersection. There are three cases depending on the coordinates of $p$.

\begin{enumerate} [label=Case \arabic*:,leftmargin=0pt,itemindent=*] 
	
	\item $p =(b,\infty),b\in \mathbb{R}$. Then $b_1=b_2=-b$.   If $q= (x_q,y_q) \in \mathbb{R}^2$, then the proof follows from Lemma \ref{M2full}. 
	Otherwise, $q= (\infty,c)$ so that $c_1=c_2=c$. In this case, $a_1 \ne a_2$ and the proof follows also from Lemma \ref{M2full}.

	\item $p =(\infty, c), c\in \mathbb{R}$. This case can be treated similarly to the previous case. 
	
	\item $p=(x_p,y_p) \in \mathbb{R}^2$. By symmetry, we may also assume that $q=(x_q,y_q) \in \mathbb{R}^2$. It is sufficient to show that if $p,q$ are on convex branches of $C$ and $D$, then $C$ and $D$ have no other intersections.   Comparing the cases in  Lemmas \ref{M2full} and \ref{obviouscalculation}, we have $b_1 \ne b_2$ and $c_1 \ne c_2$, which show that $C$ and $D$ have no intersections at  infinity. Also, the equation 
	$$a_1f_1(x+b_1)+c_1=a_2f_1(x+b_2)+c_2$$
	has two solutions $x_p,x_q$.   By Lemma \ref{obviouscalculation}, the equation 
	$$ -a_1f_2(-x-b_1)+c_1= -a_2f_2(-x-b_2)+c_2$$ 
	has no solution. This implies the concave branch of $C$ does not intersect the concave branch of $D$.
	We note that the convex branch of $C$ cannot intersect the concave branch of $D$ and vice versa. This completes the proof. \qedhere	 
	\end{enumerate}
\end{proof} 

\subsection{Axiom of Touching, existence}	

We say that two distinct elements $C,D$ of $\mathcal{C}^{-}(f_1,f_2)$ \textit{touch}  at $p$ if  $C \cap D=\{p\}$. As a preparation for the proof of the main theorem of this subsection, we have the following. 
\begin{lemma} \label{tangentimpliestouch}Let $a_1,a_2>0, b_1,b_2,c_1,c_2 \in \mathbb{R}, (a_1,b_1,c_1) \ne (a_2,b_2,c_2)$. If one of the following conditions holds, then $D_1=\overline{f_{a_1,b_1,c_1}}$ and $D_2=\overline{f_{a_2,b_2,c_2}}$  touch at a point $p$.
	
	\begin{enumerate}
		
		\item $p=(b,\infty)$, $a_1=a_2$, $b_1=b_2=-b$ and $c_1 \ne c_2$. 
		
		\item $p=(\infty,c)$, $a_1=a_2$, $b_1 \ne b_2$, and $c_1=c_2=c$.
		
		\item  $p = (x_p,y_p) \in \mathbb{R}^2$, $f_{a_1,b_1,c_1}(x_p)=f_{a_2,b_2,c_2}(x_p)=y_p$ and $f'_{a_1,b_1,c_1}(x_p)=f'_{a_2,b_2,c_2}(x_p)$.
	\end{enumerate}
	
	\begin{proof}  \begin{enumerate}[leftmargin=0pt,itemindent=*]
			\item Assume that the first condition holds.  Note that $p \in D_1 \cap D_2$.   By Lemma \ref{M2full}, $D_1$ and $D_2$ have no intersection on their convex and concave branches. They cannot have an intersection of the form $(\infty,c)$ either, since $c_1 \ne c_2$. Therefore, $p$ is the only common point of $D_1$ and $D_2$.  
			A similar conclusion can be made for the second condition. 
			
			\item Assume that the third condition holds. Since $f_{a_1,b_1,c_1}(x_p)=f_{a_2,b_2,c_2}(x_p)=y_p$, the point $p$ is an intersection of $D_1$ and $D_2$.    If $p$ is on  the convex branch of $D_1$ and the concave branch of $D_2$, then it is easy to check that  $D_1$ and $D_2$ have no other intersection.
			Consider the case $p$ is on the convex branches of $D_1$ and $D_2$. Then the function $\widecheck{f}: (\max \{ -b_1,-b_2\},+\infty) \rightarrow \mathbb{R}$ defined by
			$$\widecheck{f}(x)=a_1f_1(x+b_1)+c_1-a_2f_1(x+b_2)-c_2$$ 
			has at least one root $x_p$ such that $\widecheck{f}'(x_p)=0$. 
			
			If $b_1 =b_2, c_1\ne c_2$ or $b_1\ne b_2, c_1=c_2$, then  from Lemma \ref{M2full} we obtain a contradiction. Hence  $b_1 \ne b_2, c_1 \ne c_2$. By Lemma \ref{obviouscalculation},  $p$ is the only finite intersection of $D_1$ and $D_2$. 
			Also,  $D_1$ and $D_2$ have no intersection at infinity, and so $p$ is the only common point of $D_1$ and $D_2$. This completes the proof. \qedhere
		\end{enumerate}
	\end{proof}
	
\end{lemma}


\begin{lemma} \label{touchatinfinity1} Let $C=\overline{f_{a_0,b_0,c_0}}$ and  $p= ( -b_0, \infty)$. Let $q$ be a point such that $q \not \in C, q \not \parallel p$. Then there exist $a_1>0,b_1,c_1 \in \mathbb{R}$ such that $D=\overline{f_{a_1,b_1,c_1}}$  contains $p,q$ and touches $C$ at $p$.
	\begin{proof} Let $a_1=a_0, b_1=b_0$. Depending on the position of $q$, we let  $c_1$ be described as follows. 
		
		\begin{enumerate} [label=Case \arabic*:,leftmargin=0pt,itemindent=*] 
			\item $q=(\infty,y_q)$. 
			  Let $ c_1=y_q$. 
			\item $q=(x_q,y_q), x_q>-b_1$. Let  $c_1=y_q-a_1f_1(x_q+b_1)$. 
				\item $q=(x_q,y_q), x_q<-b_1$. Let  $c_1=y_q+a_1f_2(-x_q-b_1)$.  
		\end{enumerate}
		
		We note that the condition $q \not \in C$ implies that $c_1 \ne c_0$ in each of the cases above. Let $D=\overline{f_{a_1,b_1,c_1}}$. Then $D$ contains $p,q$, and by Lemma \ref{tangentimpliestouch}, touches $C$ at $p$. \qedhere
		
	\end{proof}
\end{lemma}

\begin{lemma} \label{touchatinfinity2} Let  $C=\overline{f_{a_0,b_0,c_0}}$ and $p= ( \infty, c_0)$. Let $q$ be a point such that $q \not \in C, q \not \parallel p$. Then there exist $a_1>0,b_1,c_1 \in \mathbb{R}$ such that $D=\overline{f_{a_1,b_1,c_1}}$ contains $p,q$ and touches $C$ at $p$.
	\begin{proof} Let $a_1=a_0, c_1=c_0$. There are three cases depending on the position of $q$.
		
		\begin{enumerate} [label=Case \arabic*:,leftmargin=0pt,itemindent=*] 
			
			\item $q= (x_q,\infty)$.  
			Let $ b_1=x_q$. We note that $b_1= x_q \ne b_0$, since   $q \not \in C$.
			\item $q=(x_q,y_q) \in \mathbb{R}^2 , y_q>c_1$. Since $f_1$ is surjective on $\mathbb{R}^+$, there exists $b_1 \in (-x_q,\infty)$ such that $a_1f_1(x_q+b_1)+c_1=y_q$. Since $q \not \in C$, it follows that $y_q \ne  a_0f_1(x_q+b_0)+c_0$, 
			which implies $f_1(x_q+b_1) \ne f_1(x_q+b_0)$.
			In particular, $b_1 \ne b_0$ since $f_1$ is injective  on $\mathbb{R}^+$.
			\item $q=(x_q,y_q), y_q<c_1$. This is similar to Case 2. 
		\end{enumerate}
		Let $D=\overline{f_{a_1,b_1,c_1}}$. Then $D$ contains $p,q$, and by Lemma \ref{tangentimpliestouch}, touches $C$ at $p$. \qedhere
	\end{proof}
	
\end{lemma}

\begin{lemma} \label{constructivetouch} Let $p=(x_p,y_p) \in \mathbb{R}^2$ and $s<0$. Let $q$ be a point such that  $q \not \parallel p$. Then exactly one of the following is true. 
	\begin{enumerate}
		\item There exists $t \in \mathbb{R}$ such that $D=\overline{l_{s,t}}$ contains $p$ and $q$.
		\item There exist $a_1>0,b_1,c_1 \in \mathbb{R}$ such that $D=\overline{f_{a_1,b_1,c_1}}$ contains $p,q$  and $f'_{a_1,b_1,c_1}(x_p)=s.$
	\end{enumerate}
	
	\begin{proof} We can assume $p=(0,0)$. If $q$ is on the line $y=sx$ or $q=(\infty,\infty)$, then  $D= \overline{l_{s,0}}$ contains $p$ and $q$. 
	We now consider other cases. 
		
		\begin{enumerate} [label=Case \arabic*:,leftmargin=0pt,itemindent=*] 		
%
			\item $q= (x_q,\infty)$, $x_q \ne 0$. Let $b_1=-x_q$. If $x_q<0$, we consider the system
			\begin{equation*} 
			\begin{cases}
			0= af_1(b_1)+c\\
			s=af_1'(b_1)
			\end{cases}
			\end{equation*}
			in variables $a,c$. The solution is $a_1= \dfrac{s}{f_1'(b_1)}, c_1=-a_1f_1(b_1)$. Then $D=\overline{f_{a_1,b_1,c_1}}$ contains $p,q$  and satisfies $f'_{a_1,b_1,c_1}(x_p)=s.$ The case $x_q>0$ is similar. 
			
			
			\item $q= (\infty,y_q), y_q \ne 0$. Let $c_1=y_q$.  If $y_q<0$, then we consider the system
			\begin{equation} \label{etoucheqn-1}
			\begin{cases}
			0= af_1(b)+c_1\\
			s=af_1'(b)
			\end{cases}
			\end{equation}
			in variables $a,b$. 
			Eliminating $a$, we obtain $
			-s/c_1=  f_1'(b)/f_1(b). $
			Let $g: (0,+\infty) \rightarrow \mathbb{R}$ be defined by
			$$
			g(b)= \dfrac{f_1'(b)}{f_1(b)}.
			$$
			Since $f_1$ is continuously differentiable, $g$ is continuous. By Lemma \ref{stronglyhyperboliclimits}, $\lim\limits_{b \rightarrow +\infty} g(b)= 0$ and  $g$ is unbounded below.  
			As $-s/c_1 <0$, by the IVT,  (\ref{etoucheqn-1}) has a solution. Then $D=\overline{f_{a_1,b_1,c_1}}$ contains $p,q$ with $p$ on its convex branch, and $f'_{a_1,b_1,c_1}(x_p)=s$. 
			Similarly, when $y_q>0$,   $D =\overline{f_{a_1,b_1,c_1}}$ contains $p,q$ with $p$ on its concave branch, and $f'_{a_1,b_1,c_1}(x_p)=s$.
%
			\item $q = (x_q,y_q) \in \mathbb{R}^2$, where $x_q>0, sx_q<y_q<0$, or  $x_q<0, 0<sx_q<y_q$. We claim that there exists $\overline{f_{a_1,b_1,c_1}}$ containing $p,q$ on its convex branch and  $f'_{a_1,b_1,c_1}(x_p)=s$, that is, the system 
			$$ 
			\begin{cases}
			0= af_1(b)+c\\
			y_q= af_1(x_q+b)+c\\
			s=af_1'(b)
			\end{cases}
			$$
			has a solution $a_1,b_1,c_1$. From the first equation, we get $c=-af_1(b)$. Substituting this into the remaining equations and rearranging, we have
			\begin{equation} \label{etoucheqn1}
			a=\dfrac{y_q}{f_1(x_q+b)-f_1(b)}=\dfrac{s}{f_1'(b)}.
			\end{equation}

			Let $g: (\max \{-x_q,0\}, +\infty) \rightarrow \mathbb{R}$ be defined by
			$$
			g(b)= \dfrac{f_1(x_q+b)-f_1(b)}{f_1'(b)}. 
			$$
			We consider the case $x_q>0$,$ sx_q<y_q<0$. Note that $g(b)$ is continuous and $\lim\limits_{b \rightarrow 0+} g(b)=0$. 
			By Lemma \ref{stronglyhyperboliclimits},
			$
			\lim\limits_{b \rightarrow +\infty} g(b) =  x_q.
			$	
			By the IVT, there exists $b_1$ such that  $g(b_1) = \dfrac{y_q}{s} \in (0,x_q)$. This shows that (\ref{etoucheqn1}) has a solution.
			The case $x_q<0, 0<sx_q<y_q$ is similar. The claim follows.
			
			\item $q = (x_q,y_q) \in \mathbb{R}^2$, where $x_q<0,y_q<0$.
			 It can be shown that there exists $\overline{f_{a_1,b_1,c_1}}$ containing $p$ on its convex branch, $q$ on its concave branch, and $f'_{a_1,b_1,c_1}(x_p)=s$.
			
			\item  $q = (x_q,y_q) \in \mathbb{R}^2$, where $0<x_q,y_q<s x_q<0$, or  $x_q<0, 0<y_q<sx_q$. It can be shown that there  exists $\overline{f_{a_1,b_1,c_1}}$ containing $p,q$ on its on concave branch  and $f'_{a_1,b_1,c_1}(x_p)=s$.  
			
			\item 	 $q = (x_q,y_q) \in \mathbb{R}^2$, where  $x_q>0,y_q>0$. It can be shown that there exists $\overline{f_{a_1,b_1,c_1}}$ containing $p$ on its concave branch and $q$ on its convex branch,  and $f'_{a_1,b_1,c_1}(x_p)=s$.  
			\qedhere
		\end{enumerate}
		
	\end{proof}
\end{lemma}

\begin{theorem}[Axiom of Touching, existence]  \label{existencetouching} Given $C \in \mathcal{C}^{-}(f_1,f_2)$ and two points $p \in C, q  \not \in C$ such that $q \not \parallel p$, there exists $D \in \mathcal{C}^{-}(f_1,f_2)$ that contains $p,q$ and touches $C$ at $p$.
	
	\begin{proof}  There are four cases depending on $p$ and $C$. 
		\begin{enumerate} [label=Case \arabic*:,leftmargin=0pt,itemindent=*] 
			
			\item $ p= (\infty,\infty)$, $C= \overline{l_{s_0,t_0}}$. Then $q \in \mathbb{R}^2$. Let  $D$ be the Euclidean line goes through $q$ and parallel to $C$. Then $D$ satisfies the requirements.

			\item  $ p = (-b_0,\infty)$, $C=\overline{f_{a_0,b_0,c_0}}$. The choice of $D$ follows from Lemma \ref{touchatinfinity1}.
			
			\item  $p = (\infty,c_0)$, $C=\overline{f_{a_0,b_0,c_0}}$.  The choice of $D$ follows from Lemma \ref{touchatinfinity2}.
			\item $p= (x_p,y_p) \in \mathbb{R}^2$. Let $s$ be the slope of the tangent of $C$ at $p$. By Lemma \ref{constructivetouch}, there exist $a_1>0,b_1,c_1 \in \mathbb{R}$ such that $D=\overline{f_{a_1,b_1,c_1}}$ contains $p,q$  and satisfies $f'_{a_1,b_1,c_1}(x_p)=s.$
			Then by Lemma \ref{tangentimpliestouch}, $D$ touches $C$ at $p$. \qedhere
			
		\end{enumerate}
		
	\end{proof}
	
\end{theorem}

\subsection{Axiom of Touching, uniqueness}	
In this subsection, we prove the following theorem.
\begin{theorem}[Axiom of Touching, uniqueness]   \label{uniquenesstouching} Let $C \in \mathcal{C}^{-}(f_1,f_2)$, $p \in C$, $q \not \in C$, $q \not \parallel p$. Then there is at most one element $D \in \mathcal{C}^{-}(f_1,f_2)$ that contains $p,q$ and touches $C$ at $p$.
\end{theorem}
\begin{lemma} \label{touchimpliessametangent} Let $D_1=\overline{f_{a_1,b_1,c_1}}$ and $D_2=\overline{f_{a_2,b_2,c_2}}$ touch at a point $p$. 
	
	\begin{enumerate}
		
		\item If $p=(b,\infty)$, $b\in \mathbb{R}$, then $a_1=a_2$, $b_1=b_2=-b$, and $c_1 \ne c_2$. 
		
		\item If $p=(\infty,c)$, $c \in \mathbb{R}$, then $a_1=a_2$, $b_1 \ne b_2$, and $c_1=c_2=c$.
		
		\item If $p = (x_p,y_p) \in \mathbb{R}^2$, then $f'_{a_1,b_1,c_1}(x_p)=f'_{a_2,b_2,c_2}(x_p)$.
	\end{enumerate}
	
	\begin{proof}  \begin{enumerate}[leftmargin=0pt,itemindent=*]
			\item If $p$ is a point at infinity, then   $b_1=b_2, c_1 \ne c_2$ or $b_1 \ne  b_2, c_1 =c_2$.  Suppose for a contradiction that $a_1 \ne a_2.$ 
		Since $D_1$ and $D_2$ touch at $p$, their convex branches have no intersections. This implies that the equation
		$$a_1f_1(x+b_1)+c_1=a_2f_1(x+b_2)+c_2$$
		has no roots. By Lemma \ref{M2full}, the equation
		$$
		-a_1f_2(-x-b_1)+c_1=-a_2f_2(-x-b_2)+c_2
		$$
		has exactly one root. This implies that $D_1$ and $D_2$ have  one intersection on their concave branches, which is a contradiction. Thus, $a_1=a_2$. 
		
		\item Assume that $p = (x_p,y_p) \in \mathbb{R}^2$. Since $p$ is the only intersection of $D_1$ and $D_2$, we have $b_1 \ne b_2, c_1 \ne c_2$. We note that if $p$ is on the convex branch of $D_1$ and the concave branch of $D_2$, then $f'_{a_1,b_1,c_1}(x_p)=f'_{a_2,b_2,c_2}(x_p)$.
		
		Consider the case $p$ is on the convex branches of both $D_1$ and $D_2$. Let $\widecheck{f}: (\max \{ -b_1,-b_2\},+\infty) \rightarrow \mathbb{R}$ defined by
		$$\widecheck{f}(x)=a_1f_1(x+b_1)+c_1-a_2f_1(x+b_2)-c_2.$$ 
		
		By checking the cases in Lemma \ref{obviouscalculation}, we see that the derivative $\widecheck{f}'$ also has a root at $x_0$. This implies that $f'_{a_1,b_1,c_1}(x_p)=f'_{a_2,b_2,c_2}(x_p)$. The case $p$ on the concave branches of $D_1$ and $D_2$ is similar.
		\qedhere
	\end{enumerate} 
		
	\end{proof}
\end{lemma}


	\begin{proof}[Proof of Theorem \ref{uniquenesstouching}] Assume  that there are $D_1, D_2 \in \mathcal{C}^{-}(f_1,f_2)$ that contain $p,q$ and touch $C$ at $p$. We  will show that $D_1=D_2$.  There are four cases depending on the coordinates of $p$. 
		
		\begin{enumerate} [label=Case \arabic*:,leftmargin=0pt,itemindent=*] 
			
			\item $p =(\infty,\infty)$. Then $C,D_1,D_2$ are Euclidean lines extended by $p$. The lines  $D_1$ and $D_2$ are parallel to $C$ and go through $q$ so they must be the same line. 
			
			\item $p  = (b,\infty), b \in \mathbb{R}$.  Since $C,D_1,D_2$ contain $p$, they cannot be Euclidean lines extended by $(\infty,\infty)$. Let $C= \overline{f_{a_0,b_0,c_0}}$, $D_1=\overline{f_{a_1,b_1,c_1}}$ and $D_2=\overline{f_{a_2,b_2,c_2}}$.   It follows that  $b_1=b_2=b_0$. 
			Since the two pairs of circles $C,D_1$ and $C,D_2$ touch at $p$,  by  Lemma \ref{touchimpliessametangent},   $a_1=a_2=a_0$. 
			
			Suppose for a contradiction that $c_1\ne c_2$. Then $q$ must be either on the convex branches or the concave branches of both $D_1,D_2$. But this is impossible by the condition $a_1=a_2$ and Lemma \ref{M2full}.
			Hence  $c_1=c_2$ and $(a_1,b_1,c_1) = (a_2,b_2,c_2)$, so that $D_1=D_2$. 
			
			\item $p  = (\infty,c)$. This can be treated similarly to Case 2.

			\item $p= (x_p,y_p) \in \mathbb{R}^2$. We  only consider the nontrivial case when $C= \overline{f_{a_0,b_0,c_0}}$,  $D_1=\overline{f_{a_1,b_1,c_1}}$ and $D_2=\overline{f_{a_2,b_2,c_2}}$.  Since $C$ and $D_1$ touch at a finite point $p \in \mathbb{R}^2$, by Lemma \ref{touchimpliessametangent},  $f_{a_0,b_0,c_0}'(x_p)=f'_{a_1,b_1,c_1}(x_p)$. Similarly, we have $f_{a_0,b_0,c_0}'(x_p)=f'_{a_2,b_2,c_2}(x_p)$. It follows that $f_{a_1,b_1,c_1}'(x_p)=f'_{a_2,b_2,c_2}(x_p)$. 
			
			Suppose for a contradiction that $(a_1,b_1,c_1) \ne (a_2,b_2,c_2)$.
			By Lemma \ref{M2full}, it cannot be the case $b_1 \ne b_2, c_1=c_2$ or $b_1=b_2,c_1 \ne c_2$.
			If $b_1 \ne b_2, c_1 \ne c_2$, then we obtain a contradiction from Lemma \ref{obviouscalculation}. 
			Hence $b_1=b_2,c_1=c_2$. Then $D_1$ and $D_2$ have at least three points in common (which are $p$ and two points at infinity), and so $D_1=D_2$.  \qedhere
			
		\end{enumerate}
	\end{proof}

\section{Isomorphism classes, automorphisms and Klein-Kroll types} 

We say a flat Minkowski plane  $\mathcal{M}_f=\mathcal{M}(f_1,f_2;f_3,f_4)$ is  \textit{normalised} if $f_1(1)=f_3(1)=1$. Every flat Minkowski plane $\mathcal{M}_f$ can be described in a normalised form and we will assume that all planes  $\mathcal{M}_f$  under consideration are normalised.

Let $\text{Aut}(\mathcal{M}_f)$ be the full automorphism group of $\mathcal{M}_f$. Let $\Sigma_f$ be the connected component of  $\text{Aut}(\mathcal{M}_f)$. From the definition of $\mathcal{M}_f$, the group $\Phi_\infty$ is contained in $\Sigma_f$.  

\begin{lemma} \label{isomorphismnotfixinfinity} Let $\phi$ be an isomorphism between two flat Minkowski planes $\mathcal{M}_f=\mathcal{M}(f_1,f_2;f_3,f_4)$ and $\mathcal{M}_g=\mathcal{M}(g_1,g_2;g_3,g_4)$.   If $\phi(\infty,\infty)\ne (\infty,\infty)$, then both $\mathcal{M}_f$ and $\mathcal{M}_g$ are isomorphic to the classical Minkowski plane.

	\begin{proof} Assume  that $\phi(\infty,\infty)= (x_0,y_0)$, $x_0,y_0 \in \mathbb{R} \cup \{\infty\}, (x_0,y_0) \ne (\infty,\infty)$. We first note that   $\phi \Sigma_f \phi^{-1} = \Sigma_g$.  
		If $\dim \Sigma_f=\dim \Sigma_g=3$, then since both groups $\Sigma_f$ and $\Sigma_g$ are connected and contain $\Phi_\infty$, it follows that $\Sigma_f=\Sigma_g=\Phi_\infty$, and so  $\phi \Phi_\infty \phi^{-1} = \Phi_\infty$. This is a contradiction, since  $\phi \Phi_\infty \phi^{-1}$ fixes $(x_0,y_0)$, but $\Phi_\infty$ can only fix $(\infty,\infty)$.

		If $\dim \Sigma_f=\dim \Sigma_g=4$, then by Theorem \ref{introtheorem2}, $\mathcal{M}_f$ is isomorphic to either a proper half-classical plane $\mathcal{M}_{HC}(f,id)$ or a proper generalized Hartmann plane $\mathcal{M}_{GH}(r_1,s_1;r_2,s_2)$. The former case cannot occur because the plane  $\mathcal{M}_{HC}(f,id)$ does not admit $\Phi_\infty$ as a group of automorphisms. In the latter case, we have 
		$$
		\Sigma_f = \{(x,y)  \mapsto(rx+a,sy+b) \mid r,s>0,a,b \in \mathbb{R}\},
		$$ which fixes only the point $(\infty,\infty)$. This is also a contradiction similar to the previous case.
		Therefore $\dim \Sigma_f= \dim \Sigma_g=6$, and both $\mathcal{M}_f$ and $\mathcal{M}_g$ are isomorphic to the classical Minkowski plane $\mathcal{M}_C$.
	\end{proof}
\end{lemma}

We now consider the case $\phi$ maps $(\infty,\infty)$ to $(\infty,\infty)$.  
Since translations in $\mathbb{R}^2$ are automorphisms of both planes  $\mathcal{M}_f$ and $\mathcal{M}_g$, we can also assume  $\phi$ maps $(0,0)$ to $(0,0)$. Since $\phi$ induces an isomorphism between the two Desarguesian derived planes,   we can represent $\phi$ by matrices in $\text{\normalfont GL}(2,\mathbb{R})$.  Since $\phi$ maps parallel classes to parallel classes, $\phi$ is of the form $\begin{bmatrix}
u & 0 \\
0 & v
\end{bmatrix}$ or $\begin{bmatrix}
0 & u \\
v & 0  
\end{bmatrix}$,
for $u,v \in \mathbb{R}\backslash \{ 0 \}$.  We rewrite 
$$
\begin{bmatrix}
u & 0 \\
0 & v
\end{bmatrix}
=
\begin{bmatrix}
r & 0 \\
0 & s  
\end{bmatrix}\cdot A,
$$
 and
$$
\begin{bmatrix}
0 & u \\
v & 0  
\end{bmatrix}
=
\begin{bmatrix}
r & 0 \\
0 & s  
\end{bmatrix} \cdot
\begin{bmatrix}
0 & 1 \\
1 & 0  
\end{bmatrix}\cdot A,
$$
where $r= |u|,s=|v|$, and
\begin{align*}
 A \in \mathfrak{A} = 
\left\{ 
\begin{bmatrix}
1 & 0 \\
0 & 1  
\end{bmatrix},
\begin{bmatrix}
1 & 0 \\
0 & -1  
\end{bmatrix},
\begin{bmatrix}
-1 & 0 \\
0 & 1  
\end{bmatrix},
\begin{bmatrix}
-1 & 0 \\
0 & -1  
\end{bmatrix}
\right\}.
\end{align*}

Let the matrices in $\mathfrak{A}$ be $A_1,A_2,A_3,A_4$, respectively.  
Then $\phi$ is a composition of maps of the form  $\begin{bmatrix}
r & 0 \\
0 & s  
\end{bmatrix}$,
 $\begin{bmatrix}
0 & 1 \\
1 & 0  
\end{bmatrix}$ or $A_i$, where $i =1 ..4$. To describe $\phi$, it is sufficient to describe  each of these maps.
This is done  in Lemmas \ref{isomorphismrs}, \ref{isomorphismflip}, \ref{isomorphismA}.

\begin{lemma}  \label{isomorphismrs}
	
	For $r,s>0$, let $\phi:  \mathcal{P} \rightarrow \mathcal{P} : (x,y) \mapsto (rx,sy)$. If $\phi$ is an isomorphism between $\mathcal{M}_f$ and $\mathcal{M}_g$, then 
	$$ (I) = 
	\begin{cases} 
	g_1(r)f_1 \left(\dfrac{1}{r}\right)=1, $ $g_2(r)f_2 \left(\dfrac{1}{r}\right)=1 \\
	g_1(x)=\dfrac{f_1(x/r)}{f_1(1/r)}\\ 
	g_2(x)=\dfrac{f_2(x/r)}{f_1(1/r)}\\ 
	g_3(x)=\dfrac{f_3(x/r)}{f_2(1/r)}\\
	g_4(x)=\dfrac{f_4(x/r)}{f_2(1/r)}
	
	\end{cases}
	$$ holds. Conversely, if there exists $r>0$ such that $(I)$ holds, then for every $s>0$, the map
	$\phi: \mathcal{P} \rightarrow \mathcal{P} : (x,y) \mapsto (rx,sy)$
	is an isomorphism between $\mathcal{M}_f$ and $\mathcal{M}_g$.
	
	\begin{proof} We only prove the equations for $g_i$, where $i=1,2$, as the cases $i=3,4$ are similar. Under  $\phi$,  the convex branch
		$\{(x,f_1(x)) \mid x>0 \}$
	of the circle $\{(x,f_1(x))\}$ is mapped 	onto 
		$\{ (x,sf_1(x/r)) \mid x>0\}.$ On the other hand, since $r,s>0$ (so that points in the first quadrant are mapped onto the first quadrant) and $\phi$ maps circles to circles, the convex branch 
		is mapped onto 
		$\{ (x,ag_1(x)) \mid x>0\},$ for some $a>0$. This implies $$sf_1(x/r)=ag_1(x)$$ for all $x>0$. For $x=1$ and $x=r$, the normalised condition gives $$f_1\left(\dfrac{1}{r}\right)=\dfrac{1}{g_1(r)}=\dfrac{a}{s},$$  so that $g_1(r)f_1 \left(\dfrac{1}{r}\right)=1$ and $g_1(x)=\dfrac{f_1(x/r)}{f_1(1/r)}$. Similarly, for the concave branch, we obtain
		$$sf_2\left(\dfrac{x}{r}\right)= ag_2(x),$$ for $x>0$. This gives $g_2(r)f_2 \left(\dfrac{1}{r}\right)=1$ and  $g_2(x)=\dfrac{f_2(x/r)}{f_1(1/r)}$. 
		
		The converse direction is easily verified.
	\end{proof}
	
\end{lemma}

\begin{lemma} \label{isomorphismflip} Let $\phi: \mathcal{P} \rightarrow \mathcal{P}: (x,y) \mapsto (y,x)$. If $\phi$ is an isomorphism between $\mathcal{M}_f$ and $\mathcal{M}_g$, then 
	$$ (II) = 
	\begin{cases}  
	g_1=f_1^{-1}\\ 
	g_2=f_2^{-1}\\ 
	g_3=\dfrac{1}{f_4^{-1}(1)}f_4^{-1}\\
	
	g_4 = \dfrac{1}{f_4^{-1}(1)}f_3^{-1}
	\end{cases}
	$$
	
	holds. 
\end{lemma}

\begin{lemma}\label{isomorphismA}
	For $A \in \mathfrak{A}$,  let 
	$
	\phi: \mathcal{P} \rightarrow \mathcal{P}:  (x,y) \mapsto (x,y)\cdot A.
	$
If $\phi$ is an isomorphism between $\mathcal{M}_f$ and $\mathcal{M}_g$, then one of the following occurs. 
	\begin{enumerate}
		\item If $A=A_1$, then $ (A1) = 
		\begin{cases}  
		g_1=f_1\\ 
		g_2=f_2\\ 
		g_3=f_3\\
		g_4=f_4
		\end{cases}
		$ holds.
		
		\item If $A=A_2$, then
		$ (A2) = 
		\begin{cases}  
		g_1=\dfrac{1} {f_4(1)}f_4\\ 
		g_2=\dfrac{1} {f_4(1)}f_3\\ 
		g_3 =\dfrac{1}{f_2(1)}f_2\\
		
		g_4 = \dfrac{1}{f_2(1)}f_1
		\end{cases}
		$ holds. 
		
		\item If $A=A_3$, then		
		$ (A3) = 
		\begin{cases}  
		g_1=f_3\\ 
		g_2=f_4\\ 
		g_3=f_1\\
		g_4=f_2
		\end{cases}
		$ holds.
		
		\item If $A=A_4$, then
		$ (A4) = 
		\begin{cases}  
		g_1=\dfrac{1}{f_2(1)}f_2\\ 
		g_2=\dfrac{1}{f_2(1)}f_1\\ 
		g_3=\dfrac{1}{f_4(1)}f_4\\
		g_4=\dfrac{1}{f_4(1)}f_3
		\end{cases}
		$	holds. 
	\end{enumerate} 
\end{lemma}

The proofs for Lemmas  \ref{isomorphismflip} and \ref{isomorphismA} can be verified by direct calculations. 

\begin{theorem} \label{isomorphictoclassical}
	A normalised flat Minkowski plane  $\mathcal{M}_f$ is isomorphic to the classical Minkowski plane $\mathcal{M}_C$  if and only if $f_i(x)=1/x$, for $i=1..4$. 
	
	\begin{proof} The ``if" direction is straightforward. For the converse direction,  let $g_i(x)=1/x$, where	 $i=1..4.$ Then  $\mathcal{M}_g=\mathcal{M}(g_1,g_2;g_3,g_4)$ is isomorphic to the classical flat Minkowski plane  and 
		$$\Sigma_g \cong  \langle PGL(2,\mathbb{R}) \times PGL(2,\mathbb{R}), \{ (x,y) \mapsto (y,x) \} \rangle.$$
		Assume    $\mathcal{M}_f$ and  $\mathcal{M}_g$ are isomorphic.	Let $\phi$ be an isomorphism between $\mathcal{M}_f$ and $\mathcal{M}_g$.
		Since $\Sigma_g$ is transitive on the torus, for $(x_0,y_0) \in \mathbb{S}_1 \times \mathbb{S}_1$, there exists an automorphism $\psi \in \Sigma_g$ such that  $\psi((x_0,y_0))=(\infty,\infty)$.            
		If $\phi((\infty,\infty))=(x_0,y_0) \ne (\infty,\infty)$, then $\psi\phi$ is an isomorphism between $\mathcal{M}_f$ and $\mathcal{M}_g$ that maps $(\infty,\infty)$ to $(\infty,\infty)$. 
		From Lemmas   \ref{isomorphismrs}, \ref{isomorphismflip}, and \ref{isomorphismA},  $f_i(x)=1/x$. 
	\end{proof}
\end{theorem}

Let $\mathfrak{A}'=\begin{bmatrix}
0 & 1 \\
1 & 0 
\end{bmatrix}\mathfrak{A}$. We now determine isomorphism classes of the planes $\mathcal{M}_f$.

\begin{theorem} \label{isomorphismclasses}	Up to isomorphisms in $\mathfrak{A}$ and $\mathfrak{A}'$, two flat Minkowski planes $\mathcal{M}_f$ and $\mathcal{M}_g$ are isomorphic if and only if  there exists $r>0$ such that $(I)$ holds. (cp. Lemma \ref{isomorphismrs}). 
	
	\begin{proof}  Let $\phi$ be an isomorphism between $\mathcal{M}_f$ and $\mathcal{M}_g$.  
		If both $\mathcal{M}_f$ and $\mathcal{M}_g$ are isomorphic to the  classical flat Minkowski plane, then the proof follows from Theorem \ref{isomorphictoclassical}.  Otherwise,  from Lemma \ref{isomorphismnotfixinfinity}, $\phi$ maps $(\infty,\infty)$ to $(\infty,\infty)$. Up to isomorphisms in $\mathfrak{A}$ and $\mathfrak{A}'$, $\phi$ has the form  $\begin{bmatrix}
		r & 0 \\
		0 & s  
		\end{bmatrix}$, for some $r,s>0$. The proof now follows from Lemma \ref{isomorphismrs}. 
	\end{proof}
\end{theorem}

From Theorem \ref{isomorphismclasses}, we can determine when a flat Minkowski  plane $\mathcal{M}_f$ is isomorphic to    a generalised Hartmann plane  (cp. Example \ref{ex:hartmann}) as follows.

\begin{theorem} \label{isomorphictoGH} Up to isomorphisms in $\mathfrak{A}$ and $\mathfrak{A}'$, a proper normalised flat Minkowski plane  $\mathcal{M}_f$ is isomorphic to  a generalised Hartmann plane $\mathcal{M}_{GH}(r_1,s_1;r_2,s_2)$  if and only if $f_1(x)=x^{-r_1},f_2(x)=s_1^{-1}x^{-r_1},  f_3(x)=x^{-r_2}, f_4(x)=s_2^{-1}x^{-r_2}$, where $r_1,s_1,r_2,s_2>0$, $(r_1,s_1,r_2,s_2) \ne (1,1,1,1)$.
\end{theorem}

From previous results, we obtain the following group dimension classification.
 
\begin{theorem}\label{SHclassification} A normalised flat Minkowski plane $\mathcal{M}_f=\mathcal{M}(f_1,f_2;f_3,f_4)$ has group dimension
	\begin{enumerate}[label=$ $]
		
		\item 6 if and only if  $f_i(x)=1/x$, for $i=1..4$;
		
		\item 4 if and only if $f_1(x)=x^{-r_1}, f_2(x)=s_1^{-1}x^{-r_1},  f_3(x)=x^{-r_2}, f_4(x)=s_2^{-1}x^{-r_2}$,  where $r_1,s_1,r_2,s_2 \in \mathbb{R}^+$, and $(r_1,s_1,r_2,s_2) \ne (1,1,1,1)$;
		
		\item 3 in all other cases. 
	\end{enumerate}
	
	\begin{proof} Let $n$ be the group dimension of  $\mathcal{M}_f$. Then $n \ge 3$, since it admits the group $\Phi_\infty$ as a group of automorphisms. 
		By Theorem \ref{introtheorem2}, $n \ge5$ if and only if $n=6$ if and only if $\mathcal{M}_f$ is isomorphic to the classical Minkowski plane. The form of $f_i$ follows from Theorem \ref{isomorphictoclassical}.  Furthermore, $n=4$  if and only if $\mathcal{M}_f$ is  isomorphic to either a proper half-classical plane $\mathcal{M}(f,id)$ or a proper generalized Hartmann plane $\mathcal{M}(r_1,s_1;r_2,s_2)$. The former case cannot occur, however,  since the plane  $\mathcal{M}(f,id)$ does not admit $\Phi_\infty$ as a group of automorphisms. In the latter case, the form of $f_i$ follows from Theorem \ref{isomorphictoGH}. 
	\end{proof}
\end{theorem}
 
Regarding the Klein-Kroll types, we obtain the following.
\begin{theorem} \label{SHkleinkroll} A flat Minkowski plane $\mathcal{M}_f$ has Klein-Kroll type
	\begin{enumerate}[label=$ $]
		
		\item \text{\normalfont VII.F.23} if it is isomorphic to the classical flat Minkowski plane;
		
		\item \text{\normalfont III.C.19} if it is isomorphic to a  generalised Hartmann plane $\mathcal{M}_{GH}(r, 1; r, 1), r \ne 1$;
		
		\item \text{\normalfont III.C.1 }  in all other cases.
	\end{enumerate}
	\begin{proof}  Let $p=(\infty,\infty)$.  The connected component $\Sigma$ of the full automorphism group  $\text{Aut}(\mathcal{M}_f)$ has a subgroup 
		$  \{\mathbb{R}^2 \rightarrow \mathbb{R}^2: (x,y) \mapsto (x+a,y+b) \mid a,b \in \mathbb{R} \}
		$
		of Euclidean translations, which is transitive on $\mathcal{P} \backslash \{ [p]_+ \cup [p]_- \}$.   
		Hence, a   plane  $\mathcal{M}_f$ has Klein-Kroll type at least III.C. 
		
		By Theorem \ref{possiblekleinkroll}, the only possible types are 	III.C.1, III.C.18, III.C.19,    or VII.F.23. By Theorem \ref{specifickleinkroll}, if $\mathcal{M}_f$ is of type III.C.18, then it  is isomorphic to  an Artzy-Groh plane with group dimension $3$. In this case,
		\[ 
		\Sigma =  
		\{ (x,y) \mapsto (ax+b,ay+c) \mid a,b,c\in \mathbb{R}, a>0\}.   
		\]
		On the other hand, $\Sigma$ contains the group $\Phi_\infty$ and since $\dim \Sigma = \dim \Phi_\infty$, it follows that $\Sigma = \Phi_0$, which is  a contradiction. The proof now follows from Theorem \ref{specifickleinkroll}. \qedhere

	\end{proof}
\end{theorem}

  \paragraph*{Acknowledgements} 
  This work was supported by a University of Canterbury Doctoral Scholarship and  partially  by UAEU grant G00003490.

\printbibliography


\end{document}